\documentclass[a4paper,10pt]{article} 

\usepackage[utf8]{inputenc} 

\usepackage[margin=1.25in]{geometry} 
\usepackage[english]{babel}
\usepackage{csquotes}
\usepackage{graphicx} 

\usepackage{amssymb}
\usepackage{stmaryrd}
\usepackage{xcolor} 
\usepackage[bookmarks=false]{hyperref} 

\usepackage[symbol]{footmisc}

\usepackage{booktabs} 
\usepackage{array} 
\usepackage{paralist} 
\usepackage{verbatim} 
\usepackage{subfig} 

\usepackage{fancyhdr} 
\usepackage{sectsty}
\allsectionsfont{\sffamily\mdseries\upshape} 

\usepackage[nottoc,notlof,notlot]{tocbibind} 
\usepackage[titles,subfigure]{tocloft} 

\usepackage{times}
\usepackage{amsmath, amsfonts, bbm, dsfont, mathrsfs}
\usepackage{xurl}

\usepackage[
backend=biber,
style=alphabetic
]{biblatex}
\bibliography{ZwB.bib}

\usepackage{amsthm}
\usepackage{mathtools}
\usepackage{tikz-cd} 
\usepackage{xfrac} 
\usepackage{enumitem}

\newtheorem{theo}{Theorem}[section]
\newtheorem{prop}[theo]{Proposition}
\newtheorem{lemma}[theo]{Lemma}
\newtheorem{coro}[theo]{Corollary}

\theoremstyle{definition}
\newtheorem{defi}[theo]{Definition}
\newtheorem{rem}[theo]{Remark}

\newcommand{\Z}{\mathbb{Z}}
\newcommand{\N}{\mathbb{N}}

\newcommand{\sep}{\mathrm{sep}}

\newcommand{\gl}[2]{\mathrm{GL}_{#1}(#2)}

\newcommand{\qp}{\mathbb{Q}_p}
\newcommand{\zp}{\mathbb{Z}_p}
\newcommand{\qptimes}{\mathbb{Q}^{\times}_p}
\newcommand{\zptimes}{\mathbb{Z}^{\times}_p}

\newcommand{\hcal}[1]{\mathcal{#1}}
\newcommand{\hrm}[1]{\mathrm{#1}}
\newcommand{\und}[1]{\underline{#1}}
\newcommand{\ov}[1]{\overline{#1}}
\newcommand{\hbb}[1]{\mathbb{#1}}

\newcommand{\cg}{\mathcal{G}}
\newcommand{\ch}{\mathcal{H}}

\newcommand{\sms}{\hrm{ss}}

\newcommand{\detale}[2]{\mathrm{Mod}^{\mathrm{\acute{e}t}}(#1,#2)}
\newcommand{\detaleproj}[2]{\mathrm{Mod}_{\mathrm{prj}}^{\mathrm{\acute{e}t}}\left(#1,#2\right)}

\newcommand{\modulo}[1]{\,\, \mathrm{mod} \,\, #1}

\newcommand{\colim}[1]{\mathop{\mathrm{colim}}}

\newcommand{\ord}{\mathrm{ord}}

\author{Nataniel Marquis}

\usepackage[T1]{fontenc}

\begin{document}

		\vspace{50ex}
		

		\hrulefill
		\begin{center}\bfseries \huge \textbf{Implications of Breuil-Herzig-Hu-Morra-Schraen's conjectures on Z\'abr\'adi's functor}
		\end{center}
		%
		\hrulefill

		\vspace*{0.5cm}

		\begin{center}\Large
			Nataniel Marquis\footnote{\underline{n.marquis@uni-muenster.de}. Universität Münster, Mathematisches Institut, Orléans-Ring 12, 48149 Münster, Germany.}
		\end{center}

			\renewcommand{\thefootnote}{\arabic{footnote}}
			\setcounter{footnote}{0}
			
		\vspace{1.5 cm}

		\textit{\textbf{Abstract.}} \textemdash\, Let $\rho$ be a smooth $n$-dimensional representation of $\cg_{\qp}$ over $\ov{\mathbb{F}}_p$. When $\rho$ is generic and a good conjugate, the article \cite{breuil2021conjectures} introduces the notion of an admissible representation $\Pi$ of $\gl{n}{\qp}$ compatible with $\rho$. In \textit{loc. cit.}, the five authors also question whether there exists some $\Pi$ compatible with $\rho$ from which Z\'abr\'adi's functor $\mathbf{V}_{\Delta}$ (see \cite{zabradi_finiteness}) recovers a specific representation $\ov{L}^{\boxtimes}(\rho)$ of $\cg_{\qp}^{n-1}$, constructed from $\rho$. We give a range of results about how badly $\mathbf{V}_{\Delta}(\Pi)$ behaves for an arbitrary $\Pi$ satisfying some weaker compatibilities with $\rho$. In particular, when $\rho$ is reducible and $n\geq 3$, no representation $\Pi$ compatible with $\widetilde{P}_{\rho}$ can satisfy $\mathbf{V}_{\Delta}(\Pi)\simeq \ov{L}^{\boxtimes}(\rho)$.

	\vspace{1cm}
	
	\textit{\textbf{Résumé.}} \textemdash\, Soit $\rho$ une représentation de $\cg_{\mathbb{Q}_p}$ sur $\ov{\mathbb{F}}_p$, lisse et de dimension $n$. Lorsque $\rho$ est générique et un bon conjugué, l'article \cite{breuil2021conjectures} introduit la notion de compatibilité avec $\rho$ pour une représentation $\Pi$ de $\gl{n}{\qp}$ sur $\ov{\mathbb{F}}_p$ lisse et admissible. Dans \textit{loc. cit.}, les cinq auteurs s'interrogent sur l'existence d'une représentation $\Pi$, compatible avec $\rho$, pour laquelle le foncteur de Z\'abr\'adi $\mathbf{V}_{\Delta}$(see \cite{zabradi_finiteness}) produirait une certaine représentation $\ov{L}^{\boxtimes}(\rho)$ de $\cg_{\qp}^{n-1}$, construite à partir de $\rho$. Nous établissons à quel point $\mathbf{V}_{\Delta}(\Pi)$ se comporte mal pour une représentation $\Pi$ arbitraire vérifiant des conditions de compatibilité avec $\rho$ affaiblies. En particulier, lorsque $\rho$ est réductible et que $n\geq 3$, aucune représentation $\Pi$ compatible avec $\widetilde{P}_{\rho}$ ne peut vérifier $\mathbf{V}_{\Delta}(\Pi)\simeq \ov{L}^{\boxtimes}(\rho)$.

	\renewcommand*{\thefootnote}{\arabic{footnote}}
	
	\vspace{1.5cm}
	
	\section{Introduction / recollection}
	
	While the mod $p$ local Langlands correspondence is well understood for $\gl{2}{\qp}$ (see \cite{breuil1} \cite{colmez_correspondance}, etc.), the situation for $\gl{n}{\qp}$ remains at the stage of groping into specific examples, especially those arising from the global picture. One of the main issues is that the theory of smooth admissible finite length $\gl{n}{\qp}$-representations modulo $p$ lacks both a classification of irreducibles as established by \cite{barthel_livné} \cite{breuil1} \cite{breuil2} for $\gl{2}{\qp}$ and the study of extensions as in \cite[\S VII]{colmez_correspondance}. While partial answers can be found both in \cite{herzig_class} for the classification and in \cite{Hauseux_2016} for the extensions' study, neither addresses supercuspidal representations.
	
	Let $\rho$ be an $n$-dimensional representation of $\cg_{\qp}$ over an algebraic extension $\mathbb{F}|\hbb{F}_p$. In \cite{breuil2021conjectures}, five authors have formulated some conjectures about the shape of representations associated with $\rho$ by the global picture, when the latter is generic and a good conjugate. More precisely, they define (see \cite[Def. 2.4.2.7]{breuil2021conjectures}) the notion of a smooth admissible representation of $\gl{n}{\qp}$ over $\mathbb{F}$ compatible with $\rho$. The number of Jordan-Hölder components of such representation $\Pi$ are prescribed by the minimal standard parabolic $P_{\rho}$ whose $\hbb{F}$-points contain the image of $\rho$; this $P_{\rho}$ also determines how these components fit in the classification of \cite{herzig_class} (see \cite[Def. 2.4.1.5]{breuil2021conjectures}). The combinatorics of $\Pi$'s subquotients is dictated by the smallest algebraic subgroup $\widetilde{P}_{\rho}$ containing the Levi $M_{P_{\rho}}$ and whose $\hbb{F}$-points contain the image of $\rho$. Finally, compatibility imposes Breuil's functor $\mathbf{V}_{\xi}$ to yield explicit subquotients of $\overline{L}^{\otimes}(\rho):=\mathop{\otimes} \limits_{1\leq i <n} \Lambda^i \rho$ when applied to irreducible subquotients of $\Pi$. 
	
	Here, Breuil's functor $\mathbf{V}_{\xi}$ (see  \cite{breuil_ind_parab}) is a version of Colmez's functor $\mathbf{V}$ (see \cite[\S IV.3]{colmez_correspondance}). The functor $\mathbf{V}_{\xi}$ produces ind-finite dimensional Galois representations from smooth representations of $\gl{n}{\qp}$. We won't need its precise definition. One current issue of Breuil's functor is that we do not know whether it produces finite dimensional or even nonzero objects, except for subquotients of the principal series or when $n=2$. 
	
	\vspace{0.5 cm}
	In \cite{zabradi_gln},  G. Z\'abr\'adi developed another generalisation of $\mathbf{V}$ for a general reductive group $G$. This functor $\mathbf{V}_{\Delta}$ will be the other actor of this note. Let us recall the philosophy behind Z\'abr\'adi's functor. We first investigate the intermediate functor $\mathbf{D}_{\Delta}^{\vee}$ producing multivariable $(\varphi,\Gamma)$-modules. In our situation, $\Delta$ is the canonical set of simple roots in the canonical root system of $\mathrm{GL}_n$; it satisfies $\# \Delta=n-1$. We identify $\llbracket 1,n-1\rrbracket$ with $\Delta$ by sending $i$ to the character of the torus $\mathrm{Diag}(x_1,\ldots,x_n)\mapsto x_ix_{i+1}^{-1}$ and equip $\Delta$ with the order on integers. Combining \cite{zabradi_gln} and \cite[Th. A]{zabradi_finiteness}, the functor $\mathbf{D}_{\Delta}^{\vee}$ produces, from any finite length smooth representation $\Pi$ of $\mathrm{GL}_n(\qp)$ over $\mathbb{F}$, a finite projective topological étale $T_{+,\Delta}$-module over $$E_{\Delta}:=\bigcup_{\mathbb{F}_q \subset \mathbb{F}}\, \mathbb{F}_q \llbracket X_{\alpha} \, |\, \alpha \in \Delta\rrbracket[X_{\Delta}^{-1}].$$ Here $T_{+,\Delta}$ is the monoid of diagonal matrices with decreasing valuations and $X_{\Delta}=\prod_{\alpha\in \Delta} \, X_{\alpha}$. 
	
	If $\Pi$ has a central character, then $\qptimes \hrm{Id}_n \subset T_{+,\Delta}$ also acts via a character on $\mathbf{D}_{\Delta}^{\vee}(\Pi)$. The slightest generalisation of \cite{zabradi_equiv}, produces an equivalence $\mathbb{V}_{\Delta}$ from $\detale{T_{+,\Delta}}{E_{\Delta}}$ with $\qptimes \hrm{Id}_n$ acting by homotheties and smooth finite dimensional representations of $(\cg_{\qp,\Delta} \times \qptimes)$  over ${\mathbb{F}}$ with central character, i.e. $\qptimes$ acting via homotethies. Post-composing $\mathbf{D}_{\Delta}^{\vee}$ with this equivalence and applying a suitable duality functor provides a functor $\mathbf{V}_{\Delta}$ to finite dimensional representations of $\cg_{\qp}^{n-1}$ (see Definition \ref{defi_VD}). 
	
	Two important features are to highlight. Firstly, ending up with multivariable $(\varphi,\Gamma)$-modules seems to keep more information than Breuil's functor in the same way that a representation of $\cg_{\qp}^{n-1}$ has more information than its restriction to the diagonal $\cg_{\qp}$. Secondly, Z\'abr\'adi produces in \cite{zabradi_gln} a \linebreak$\gl{n}{\qp}$-equivariant sheaf on the flag variety $\gl{n}{\qp}/B(\qp)$ from this multivariable $(\varphi,\Gamma)$-module associated to a $\gl{n}{\qp}$-representation, smilarily to what $D \boxtimes \hbb{P}^1(\qp)$ was for $\gl{2}{\qp}$ (see \cite[\S II.1]{colmez_correspondance}). From this equivariant sheaf, \cite{zabradi_finiteness} deduces finiteness results that are lacking for Breuil's functor. Unfortunately, they do not recover more nonvanishing results than what was known for Breuil's functor.
	
	As stated in \cite[Rem. 1.2.5]{breuil2021conjectures}, the existence of Z\'abr\'adi's functor allows to formulate an improved conjecture about the $\Pi$ produced by global methods from $\rho$. Ideally, Z\'abr\'adi's functor would take the value $$\overline{L}^{\boxtimes}(\rho):=(\boxtimes_{1\leq i <n} (\Lambda^i \rho))$$ on such $\Pi$ and its restriction to the diagonal embedding of $\cg_{\qp}$ would identify with $\mathbf{V}_{\xi}(\Pi)$, which is isomorphic to $\overline{L}^{\otimes}(\rho)$. Such a conjecture fills the following small dissatisfaction: the representation $\overline{L}^{\otimes}(\rho)$ doesn't determine $\rho$, already in dimension $3$, while $\ov{L}^{\boxtimes}(\rho)$ does.

	\vspace{0.5cm}
	
	The first aim of this note is to prove that this scenario never happens in the reducible case, even when downgrading compatibility with $\rho$ to compatibility with $\widetilde{P}_{\rho}$ (see Theorem \ref{main_failure_1}). The rough idea is that $\Pi$ being compatible with $\widetilde{P}_{\rho}$ imposes a shortage of Jordan-Hölder components of $\mathbf{V}_{\Delta}(\Pi)$, in comparison to their profusion in $\ov{L}^{\boxtimes}(\rho)$. We use the same method to treat some irreducible cases (see Theorem \ref{irred_fail}).
	
	One can argue that the condition \cite[Def. 2.4.2.7 (i)]{breuil2021conjectures} imposes heavy conditions on the image of intermediate supercuspidal components by $\mathbf{V}_{\Delta}$. What would happen if we replaced the conditions (i)-(ii) in \textit{loc. cit.} by the single condition $\mathbf{V}_{\xi}(\Pi)= \overline{L}^{\otimes}(\rho)$ ? This note also gives bad behavior of $\dim \mathbf{V}_{\Delta}(\Pi)$ under this hypothesis for a toy example: we investigate the case of $\mathrm{GL}_3(\qp)$ and $\rho$ being of dimension $3$ and Loewy length $3$ over $\overline{\mathbb{F}}_p$. For the precise result, see Proposition \ref{failure_2}. Note that the bad behaviors we prove could highlight either a vanishing of Z\'abr\'adi's functor on certain supercuspidals, or a heavy lack of exactness of the same functor. Our methods do not pinpoint which failure happens, while always proving a slight exactness failure (mainly because the image of the maximal ordinary quotient of $\Pi$ has a good behavior which is not reflected on the image of $\Pi$). Proving some vanishing would be a more involved result.	
	\vspace{0.3 cm}

	\subsection*{Notations}

	In this note, the field $\mathbb{F}$ is an algebraic extension of $\mathbb{F}_p$, equipped with the discrete topology so that continuous representations of locally profinite groups are smooth, i.e. have open stabilisers. All inductions are smooth inductions.
	
	For a finite length representation $V$ of a group $G$, we denote the socle (resp. cosocle) filtration by $\hrm{soc}^i_G V$ (resp. $\hrm{cosoc}^i_G V$) and its graded pieces by $\hrm{soc}_{i,G} V$ (resp. $\hrm{cosoc}_{i,G} V$). The Loewy length is the smallest integer $i$ such that $\hrm{soc}^i_{G}=V$.
	
	\vspace{0.5 cm}
	
	Recall the Artin map, defined by $$\hrm{Art}^{-1} \, : \, \cg_{\qp}^{\hrm{ab}} \rightarrow \widehat{\qptimes}, \,\,\, \sigma \mapsto x_{\sigma} p^{n_{\sigma}}$$ where $x_{\sigma}\in \zptimes$ is characterised by $\forall \zeta\in \mu_{p^{\infty}}, \,\, \sigma(\zeta)=\zeta^x$. The Weil group $\hrm{W}_{\qp}$ verifies that \linebreak$\hrm{Art}^{-1}(\hrm{W}_{\qp}^{\hrm{ab}})=\qptimes$.
	
	We denote by $\varepsilon$ the smooth character of $\qptimes$ defined by $$\varepsilon \,  : \, \qptimes\rightarrow \mathbb{F}_p^{\times}, \,\,\, x p^n \mapsto (x \modulo{p})$$ for $x\in \zptimes$. It extends to $\widehat{\qptimes}$.
	
	We denote by $\omega$ the smooth character of $\cg_{\qp}$ defined as the composition $$ \omega\, : \,\cg_{\qp} \twoheadrightarrow \cg_{\qp}^{\hrm{ab}} \xrightarrow{\hrm{Art}^{-1}} \widehat{\qptimes} \xrightarrow{\varepsilon} \mathbb{F}_p^{\times}.$$
	
	We will denote by $\ch_{\qp}:=\cg_{\qp (\mu_{p^{\infty}})}$, a closed normal subgroup of $\cg_{\qp}$ together with a canonical identification $\cg_{\qp}/\ch_{\qp} \simeq \zptimes$ via $\sigma \mapsto x_{\sigma}$. We often call $\Gamma_{\qp}:=\zptimes$.
	
	For any finite set $\Delta$, we define $\cg_{\qp,\Delta}:= \prod_{\alpha \in \Delta} \cg_{\qp}$ and similarly for $\ch_{\qp,\Delta}$ and $\Gamma_{\qp,\Delta}$. We also have a canonical identification $\cg_{\qp,\Delta}/\ch_{\qp,\Delta} \simeq \Gamma_{\qp,\Delta}$.
	
	\vspace{0.5 cm}

	For $G=\mathrm{GL}_n$, we keep the notation $\overline{L}^{\otimes}$ introduced in \cite[\S 2.2.1]{breuil2021conjectures}: in our situation, it is the tensor product of the antisymmetric powers of the standard algebraic representation.
	
	For a profinite group $G$, let $\hrm{Rep}_{\hbb{F}} \, G$ denote the category of smooth finite length representations $V$ of $G$ over $\hbb{F}$ coming from a finite field\footnote{The hypothesis of coming from a finite field might not be automatic. It is true for irreducible principal series of $\gl{n}{\qp}$ and for admissible irreducible supersingulars of $\gl{2}{\qp}$. The proof we know rely on finite presentation of these representations, which is false for $\gl{2}{K}$ if $K\neq \qp$ (see \cite{schraen_pres} \cite{wu_pres}) and is expected to be false for $\gl{n}{\qp}$ when $n\geq 3$.}, i.e. such that there exists a smooth finite length representation $V'$ of $G$ over a finite field $\hbb{F}_q$ verifying $V\simeq \hbb{F}\otimes_{\hbb{F}_q} V'$. Maps in this category are also assumed to come from a finite field. If $G$ is profinite, smooth finite length and smooth finite dimensional representation coincide and automatically come from a finite field. When $G=\gl{n}{\qp}$, we additionally suppose that it has central character, i.e. the center $Z(G)$ acts by homotheties.

	\vspace{0.5 cm}
	
	We recall that the principal series are the representations of $\gl{n}{\qp}$ isomorphic to some induction $\hrm{Ind}_{B(\qp)}^{\hrm{GL}_n(\qp)} \chi$, where $B$ is the standard Borel and $\chi$ ranges over characters of the standard torus.
	\vspace{0.3 cm}
	
	\subsection*{Acknowledgments} Firstly, I would like to express my gratitude to Stefano Morra, with whom I had fruitful exchanges on this topic. I am also grateful to Gergely Zábradi for spotting a mistake in the first version of this note and for his glimpse on the final version. I would also like to thank Christophe Breuil for reading the final version of this work. Finally, I would like to thank Gaëtan Chenevier and Eugen Hellmann for blackboard discussions that helped to refine some arguments in this article. I also heartily thank the anonymous referee for their careful reading and comments. I wrote this article during my postdoctoral fellowship at the University of Münster, which provided me with funding for travel, enabling me to have face-to-face discussions with some of the aforementioned researchers. 
	
	\vspace{0.75cm}
	
	\subsection*{Tool and computational resource disclosure}
		
		We used ChatGPT solely for the final grammar and orthographic proofreading of this article.
	
	\vspace{0.75 cm}
	
	\tableofcontents
	
	\vspace{0.75 cm}	

	\section{Recollection about Z\'abr\'adi's functor}

	Due to the intended brevity of this note, we will keep our definitions and recollections to a minimum. The first ingredient of $\mathbf{V}_{\Delta}$ is the multivariable Fontaine equivalence proven in \cite{zabradi_equiv}. 
	
	Let $\Delta$ be a finite set, for now arbitrary.

	\begin{defi}
		When $\hbb{F}$ is finite, we define the rings $$E_{\Delta}^+:= \mathbb{F}\llbracket X_{\alpha}\, |\, \alpha \in \Delta\rrbracket \,\,\,\, \text{and}\,\,\,\, E_{\Delta}:=E_{\Delta}^+ [X_{\Delta}^{-1}],$$ where $X_{\Delta}=\prod X_{\alpha}$ and equips it with an $\hbb{F}$-linear\footnote{As $\hbb{F}$ keeps track of the representations' coefficients, the action of $\Phi_{\Delta,p}$ must fix it.} action of $$(\Phi_{\Delta,q}\times \Gamma_{\qp,\Delta}):= \prod_{\alpha \in \Delta} \varphi_{\alpha}^{\N}\times \prod_{\alpha \in \Delta} \zptimes$$ that verifies $$\forall \varphi=(\varphi_{\alpha}^{n_{\alpha}})_{\alpha}, \, \forall \beta \in \Delta, \,\, \varphi(X_{\beta})=X_{\beta}^{p^{n_{\beta}}}$$ $$\forall \gamma=(\gamma_{\alpha})_{\alpha} \in \Gamma_{\qp,\Delta}, \,  \forall \beta \in \Delta, \,\, \gamma(X_{\beta})=(1+X_{\beta})^{\gamma_{\beta}}-1.$$ In general, we define $E_{\Delta}^+$ and $E_{\Delta}$ as the colimit over finite subfields of $\hbb{F}$ of the rings above. 
		
		Let  $E:= \hbb{F}_p(\!(X)\!)$. It has the same Galois theory as $\qp(\mu_{p^{\infty}})$ (see for instance \cite[\S 3.1.8]{Fontaine_equiv}) and we choose an identification $\cg_E\simeq \ch_{\qp}$. For each $\alpha$, fix a copy $E_{\alpha}:= \hbb{F}_p(\!(X_{\alpha})\!)$ of $E$, and choose an extension of the identification $$E_{\alpha}^{\sep} \simeq E^{\sep}:=\lim \limits_{\substack{\rightarrow \\ F|E \text{ finite separable}}} F.$$ It gives a identification of $\cg_{E_{\alpha}}$ to $\cg_E$, hence to $\ch_{\qp}$. 
		
		Each $E_{\alpha}$ has an embedding into $E_{\Delta}$ and we define for any $\hbb{F}$ $$E_{\Delta}^{\sep}:= E_{\alpha}^{\sep} \otimes_{E_{\alpha}} \Big(\ldots\Big( E_{\beta}^{\sep} \otimes_{E_{\beta}} E_{\Delta}\Big)\Big)$$ As in \cite[\S 3.1]{zabradi_equiv}, it has an action of $(\Phi_{\Delta,q}\times \cg_{\qp,\Delta})$. The $\alpha$-th copy in $\ch_{\qp,\Delta}$ acts only on the factor $E_{\alpha}^{\sep}$ of the tensor product and via $\cg_{E_{\alpha}} \simeq \ch_{\qp}$. 
		
		Thanks to \cite[Lem. 3.6]{zabradi_equiv}, we have the identity $$(E_{\Delta}^{\sep})^{\Phi_{\Delta,p}}=\hbb{F}.$$ Thanks to \cite[Prop. 3.3]{zabradi_equiv}, we have an  $(\Phi_{\Delta,p}\times \Gamma_{\qp,\Delta})$-equivariant identification of rings: $$(E_{\Delta}^{\sep})^{\ch_{\qp,\Delta}}= E_{\Delta}.$$
	\end{defi}
	\vspace{0.1 cm}
	\begin{defi}[See {\cite[\S 2.2]{marquis_formalisme}}]
		For a monoid $\hcal{S}$ acting on a ring $R$, we define $\detaleproj{\hcal{S}}{R}$ to be the category of finite projective $R$-modules $D$ of constant rank, equipped with a semilinear action of $\hcal{S}$ such that $$\forall s\in \hcal{S}, \,\,  R\otimes_{s,R} D \xrightarrow{\hrm{Id} \otimes s} D$$ is an isomorphism.
	\end{defi}
	
	After this setup, we can state the multivariable cyclotomic Fontaine equivalence.
	
	\begin{theo}
		The following pair of functors \begin{align*} \mathbb{D}_{\Delta}\, :&& \hrm{Rep}_{\mathbb{F}_p}\, \cg_{\qp,\Delta} & \rightleftarrows \detaleproj{\Phi_{\Delta,p}\times \Gamma_{\Delta}}{E_{\Delta}}\, &&: \, \mathbb{V}_{\Delta} \\ &&V & \mapsto \left(E_{\Delta}^{\sep} \otimes_{\hbb{F}} V\right)^{\ch_{\qp,\Delta}} \\&& \left(E_{\Delta}^{\sep}\otimes_{E_{\Delta}} D\right)^{\Phi_{\Delta,p}} & \mapsfrom D \end{align*} are well defined. They are quasi-inverse to one another, giving an equivalence of symmetric monoidal closed categories.
	\end{theo}
	
	The right side of this equivalence is the category of multivariable cyclotomic (mod $p$) $(\varphi,\Gamma)$-modules with $\Delta$-variables.

	\begin{rem}
		To pass from \cite[Th. 3.15]{zabradi_equiv} to this theorem, one needs three more ideas. The first one is that, for $\hbb{F}$ finite, an object of $\hrm{Rep}_{\hbb{F}}\, \cg_{\qp,\Delta}$ is an object of $\hrm{Rep}_{\hbb{F}_p} \,\cg_{\qp,\Delta}$ with additional endomorphisms coding the external multiplication by $\hbb{F}$. The same is true for $\detaleproj{\Phi_{\Delta,q}\times \cg_{\qp,\Delta}}{E_{\Delta}}$ compared to its version over $\hbb{F}_p$. These structures are transferred on both sides of Z\'abr\'adi's equivalence. This handles $\hbb{F}$ being finite.
			
		For general $\hbb{F}$, one needs to show that $\hrm{Rep}_{\hbb{F}} \,\cg_{\qp,\Delta}$ and $\detaleproj{\Phi_{\Delta,q}\times \cg_{\qp,\Delta}}{E_{\Delta}}$ are colimits of the analogous categories for $\hbb{F}'$ exhausting finite subfields of $\hbb{F}$  along base change functors. For the Galois representations, it is a general fact for a smooth representation of a compact group. For $(\varphi,\Gamma)$-modules, notice that $(\Phi_{\Delta,p}\times \Gamma_{\qp,\Delta})$ is topologically finitely generated and that its action on a multivariable \linebreak$(\varphi,\Gamma)$-module is automatically continuous for some explicit topology (see \cite[Prop. 6.1.7]{phd_marquis}).
		
		The last idea is that, according to \cite[Prop. 2.2]{zabradi_equiv} \cite[Prop. 6.1.7]{phd_marquis}, the projectivity condition is automatic.
	\end{rem}
	
	We now prove a small useful fact that we did not see properly written anywhere.
	
	\begin{prop}\label{vdelta_sqcup}
		Let $\Delta=\sqcup_{i\in I} \Delta_i$ be a partition of the set $\Delta$. Let $D_i \in \detaleproj{\Phi_{\Delta_i,p}\times \Gamma_{\qp,\Delta_i}}{E_{\Delta_i}}$. Then $$D:=\mathop{\otimes}\limits_{i\in I, \,E_{\Delta}} (E_{\Delta}\otimes_{E_{\Delta_i}} D_i)$$ belongs to $\detaleproj{\Phi_{\Delta,p}\times \Gamma_{\qp,\Delta}}{E_{\Delta}}$ and  $$\mathbb{V}_{\Delta}(D)\simeq \boxtimes_{i\in I}\mathbb{V}_{\Delta_i}(D_i).$$
	\end{prop}
	\begin{proof}
		Consider $(\Phi_{\Delta_i,p}\times \Gamma_{\qp,\Delta_i})$ as a quotient of $(\Phi_{\Delta,p}\times \Gamma_{\qp,\Delta})$. This way, the map $E_{\Delta_i} \rightarrow E_{\Delta}$ is $(\Phi_{\Delta,p}\times \Gamma_{\qp,\Delta})$-equivariant for the action by inflation on the first term. Combine \cite[Prop. 3.3 (1), Prop 2.15 (4)]{marquis_formalisme} to prove that $D$ is an étale projective module.
		
		For the identification of $\hbb{V}_{\Delta}(D)$, we first deal with the case $\mathbb{F}=\mathbb{F}_p$ by adapting the strategy of \cite[Prop. 6.2.1]{phd_marquis}. By using \cite[Lem. 3.2]{zabradi_equiv} for $\Delta$ and each $\Delta_i$, then passing to the limit, we get $$E_{\Delta}^{\sep}\simeq\bigotimes_{i\in I, \,E_{\Delta}}\big(E_{\Delta_i}^{\sep}\otimes_{E_{\Delta_i}} E_{\Delta}\big).$$ Hence $$E_{\Delta}^{\sep}\otimes_{E_{\Delta}} D\simeq \bigotimes_{i\in I, \,E_{\Delta}^{\sep}}\Big(E_{\Delta}^{\sep}\otimes_{E_{\Delta_i}^{\sep}} \big(E_{\Delta_i}^{\sep} \otimes_{E_{\Delta_i}} D_i\big)\Big).$$ Combining \cite[Prop. 3.7]{zabradi_equiv} \cite[Th. 3.15]{zabradi_equiv} we get comparison isomorphisms for each $D_i$, yielding $$E_{\Delta}^{\sep}\otimes_{E_{\Delta}} D \simeq\bigotimes_{i\in I, \,E_{\Delta}^{\sep}}\Big(E_{\Delta}^{\sep}\otimes_{E_{\Delta_i}^{\sep}} \big(E_{\Delta_i}^{\sep} \otimes_{\mathbb{F}_p} \mathbb{V}_{\Delta_i}(D_i)\big)\Big)\simeq E_{\Delta}^{\sep} \otimes_{\mathbb{F}_p} \Big(\mathop{\otimes} \limits_{i\in I, \, \mathbb{F}_p} \mathbb{V}_{\Delta_i}(D_i)\Big).$$ One can check that each of our isomorphisms is $(\Phi_{\Delta,p} \times \cg_{E,\Delta})$-equivariant. Because $\boxtimes_i \, \mathbb{V}_{\Delta_i}(D_i)$ is an $\hbb{F}_p$-vector space with trivial $\Phi_{\Delta,p}$-action (see \cite[Lem. 3.6]{zabradi_equiv}), the result follows by taking \linebreak$\Phi_{\Delta,p}$-invariants.
	\end{proof}

	\begin{coro}
		For $\Delta_1\subset \Delta$ and $D_1\in \detaleproj{\Phi_{\Delta_1,p}\times \Gamma_{\qp,\Delta_1}}{E_{\Delta_1}}$, $$\mathbb{V}_{\Delta}\big(E_{\Delta}\otimes_{E_{\Delta_1}} D_1\big)= \hrm{Inf}_{\cg_{\qp,\Delta_1}}^{\cg_{\qp,\Delta}} \mathbb{V}_{\Delta_1}(D),$$ where $\hrm{Inf}$ stands for inflation.
	\end{coro}
	\begin{proof}
		Apply for $I=\{1,2\}$, $\Delta_2=\Delta\backslash \Delta_1$ and $D_2=E_{\Delta_2}$.
	\end{proof}
	
	\vspace{0.3 cm}
	The other ingredient of Z\'abr\'adi's functor is an analogue of Colmez's functor $\mathbf{D}^{\vee}$. From now on, all sets $\Delta$ will be the \textit{standard} simple roots of $\gl{n}{\qp}$. This is canonical simple roots of the \textit{standard} root system obtained from the torus of diagonal matrices. It satisfies $\# \Delta=n-1$.

	\begin{defi}
		When $\hbb{F}$ finite, the functor $\mathbf{D}_{\Delta}^{\vee}$ is the one introduced in \cite{zabradi_gln}. It takes a smooth finite length representation of $\gl{n}{\qp}$ over $\hbb{F}$ with central character and produces an étale multivariable cyclotomic $(\varphi,\Gamma)$-module with $\Delta$-variables.
		
		In general, we can define $$\mathbf{D}_{\Delta}^{\vee}(\Pi):= \lim \limits_{\substack{\longleftarrow\\ \Pi'/ \hbb{F}_q \,\, \text{s.t.}\\ \Pi=\hbb{F}\otimes_{\hbb{F}_q} \Pi'}} \hbb{F}\otimes_{\hbb{F}_q} \mathbf{D}_{\Delta}^{\vee}(\Pi')$$ with the above definition for $\mathbf{D}_{\Delta}^{\vee}(\Pi')$. One verifies that the category is filtered and all transition maps in the limit are isomorphisms.
	\end{defi}
	
	\vspace{0.3 cm}
	
We finally combine the multivariable Fontaine equivalence and Z\'abr\'adi's functor.

	\begin{defi}\label{defi_VD}
		Let $\delta^{\boxtimes}_{\Delta}$ be the character $\boxtimes_{i \in \Delta} \omega^{\sum_{j=n-i}^{n-1} j}$ of $\cg_{\qp}^{n-1}$, whose restriction to the diagonal embedding is $\delta_{\mathrm{GL_n}}$ from \cite[p. 23]{breuil2021conjectures}. We define $\mathbb{V}_{\Delta}^*:= \underline{\hrm{Hom}}(\mathbb{V}_{\Delta}, \delta^{\boxtimes}_{\Delta})$.
	\end{defi}

	\begin{defi}
		We define $\mathbf{V}_{\Delta}:= \mathbb{V}_{\Delta}^*\circ \mathbf{D}_{\Delta}^{\vee}$, contravariant functor from the category $\hrm{Rep}_{\hbb{F}} \,\gl{n}{\qp}$ to $\hrm{Rep}_{\hbb{F}} \, \cg_{\qp,\Delta}$.
		
		We define $\mathbf{V}_{\Delta,\xi}:= \hrm{Res}_{\cg_{\qp}}^{\cg_{\qp,\Delta}} \circ \mathbf{V}_{\Delta}$,  where the restriction is along the diagonal embedding.
	\end{defi}

Recall three important theorems about Z\'abr\'adi's functor, and briefly explain how to improve to $\hbb{F}$ possibly infinite.
	
	\begin{theo}\label{exact_PS}[See \cite[Th. C, Th. G]{zabradi_gln}]
		The functor $\mathbf{D}_{\Delta}^{\vee}$ is right exact. It is exact on the full subcategory $\hrm{SP}$ of $\hrm{Rep}_{\hbb{F}} \,\gl{n}{\qp}$ whose objects representations having Jordan-Hölder factors among subquotients of principal series living in $\hrm{Rep}_{\hbb{F}}\,\gl{n}{\qp}$. It is fully faithful on the full subcategory $\hrm{SP}^0$ of representations whose Jordan-Hölder factors are irreducible principal series.
		
		Dual results hold for $\mathbf{V}_{\Delta}$ as both $\mathbb{V}_{\Delta}$ and $\und{\hrm{Hom}}$ are exact.
	\end{theo}
	\begin{proof}
		Every computation can be made at finite level. In fact, any smooth character $\chi \, : \, T(\qp) \rightarrow \ov{\hbb{F}}_p$ must be defined over a finite extension as $T(\qp)$ is topologically finitely generated. Adding the fact that functions belonging to the induction $\hrm{Ind}^{\gl{n}{\qp}}_{B(\qp)} \chi$ are locally constant and that $\gl{n}{\qp}/B(\qp)$ is compact, principal series come from a finite field.
	\end{proof}
	
	\begin{theo}\label{irred_or_zero}[See {\cite[Th. C]{zabradi_finiteness}}]
		If $\Pi$ is irreducible smooth, then $\mathbf{V}_{\Delta}(\Pi)$ is either zero or irreducible.
	\end{theo}
	\begin{proof}
		Suppose that $\mathbf{V}_{\Delta}(\Pi)$ has a non trivial subextension $V$. As $\cg_{\qp,\Delta}$ is profinite, if $\mathbf{V}_{\Delta}(\Pi)$ had a non trivial subrepresentation, it would come from $V' \in \hrm{Rep}_{\hbb{F}_q} \,\cg_{\qp,\Delta}$. Suppose $\hbb{F}_q$ chosen big enough such that $\Pi$ comes from $\Pi'\in \hrm{Rep}_{\hbb{F}_q} \,\gl{n}{\qp}$ and $V'\subset \mathbf{V}_{\Delta}(\Pi')$. By contraposition of the result of Z\'abr\'adi, this implies that $\Pi'$ is not irreducible, implying the same issue for $\Pi$.
	\end{proof}
	
	\begin{rem}
		This theorem is also true if, for all big enough finite subfield $\hbb{F}_q$ of $\hbb{F}$, the representation $\Pi$ comes from an irreducible smooth $\Pi'\in \hrm{Rep}_{\hbb{F}_q} \, \gl{n}{\qp}$.
	\end{rem}
	
	\begin{theo}\label{compat_induction}[See {\cite[Prop. 3.2, Th. 3.5]{zabradi_gln}}]
		Let $P$ be a standard parabolic of $\hrm{GL}_n$. Write its Levi subgroup as the diagonal $\prod_{1\leq t\leq s} \hrm{GL}_{m_t}$. The intersection of the standard root system of $\hrm{GL}_{m_t}$ with $\Delta$ gives the standard simple roots $\Delta_t$ of $\hrm{GL}_{m_t}$. 
		
		For all $t$, let $\Pi_t\in \hrm{Rep}_{\hbb{F}} \,\gl{m_t}{\qp}$ be locally admissible. We see $\boxtimes_t \,\Pi_t$ as a representation of $P^-(\qp)$ by inflation. Then,
		
		$$\mathbf{D}_{\Delta}^{\vee}\Big(\hrm{Ind}_{P^-(\qp)}^{\hrm{GL}_n(\qp)} (\boxtimes_t\,  \Pi_t)\Big)\simeq \bigotimes_{1\leq t\leq s} \big( E_{\Delta}\otimes_{E_{\Delta_t}}\mathbf{D}_{\Delta_t}^{\vee}(\Pi_t)\big).$$
	\end{theo}
	\begin{proof}
		Similar to what need to the comments for Theorem \ref{exact_PS}.
	\end{proof}
	
	\begin{defi}\label{maps}
		We define a canonical surjection $E_{\Delta} \twoheadrightarrow E$ as the quotient by the ideal \linebreak$(X_{\alpha}-X_{\beta}\, |\, \alpha,\beta \in \Delta)$, i.e. the identification of all variables to $X$. It is $(\varphi^{\N}\times\Gamma)$-equivariant for its action on $E_{\Delta}$ via diagonal embedding of $(\varphi^{\N}\times\Gamma)$ into $(\Phi_{\Delta,p}\times \Gamma_{\qp,\Delta})$.
		
		 Using the identifications $E^{\sep}_{\alpha} \simeq E^{\sep}$ we deduce a map $$\bigotimes_{\alpha \in \Delta, \, E_{\Delta}} (E_{\alpha}^{\sep} \otimes_{E_{\alpha}} E_{\Delta}) \rightarrow E^{\sep}$$ i.e. a surjection $E_{\Delta}^{\sep}\twoheadrightarrow E^{\sep}$. It is $(\varphi^{\N}\times\cg_{\qp})$-equivariant for its action on $E_{\Delta}^{\sep}$ via diagonal embedding of $(\varphi^{\N}\times\cg_{\qp})$ into $(\Phi_{\Delta,p}\times \cg_{\qp,\Delta})$.
	\end{defi}
	
	\begin{lemma}\label{res_identif}
		There is a natural isomorphism of functors from $\detaleproj{\Phi_{\Delta,p}\times \Gamma_{\qp,\Delta}}{E_{\Delta}}$ to $\hrm{Rep}_{\hbb{F}}\, \cg_{\qp}$ $$\hrm{Res}_{\cg_{\qp}}^{\cg_{\qp,\Delta}} \circ \hbb{V}_{\Delta} \Rightarrow \hbb{V}\left(E \otimes_{E_{\Delta}} -\right).$$ Here, the restriction is along the diagonal embedding and the tensor product is correctly defined using \cite[Prop. 3.3]{marquis_formalisme} for the $(\varphi^{\N}\times \Gamma)$-equivariant canonical map.
	\end{lemma}
	\begin{proof}
		Factorising $E_{\Delta}\twoheadrightarrow E\hookrightarrow E^{\sep}$ through $E_{\Delta}^{\sep}$, we obtain a natural map $$ \left(E_{\Delta}^{\sep}\otimes_{E_{\Delta}} D\right)^{\Phi_{\Delta,p}}\rightarrow \left(E^{\sep} \otimes_{E_{\Delta}^{\sep}} \left(E_{\Delta}^{\sep}\otimes_{E_{\Delta}} D\right)\right)^{\varphi=\hrm{Id}}.$$ It is a map of $\cg_{\qp}$-representations, hence defines a natural transformation as required.
		
		To check that it is a natural isomorphism, we can forget the action of $\cg_{\qp}$. The isomorphism of comparison, proved together with multivariable cyclotomic Fontaine equivalence, implies that $$E_{\Delta}^{\sep}\otimes_{E_{\Delta}} D \simeq (E_{\Delta}^{\sep})^{\oplus \hrm{rk} D}$$ in $\detaleproj{\Phi_{\Delta,p}}{E_{\Delta}^{\sep}}$. Hence, proving that the map constructed above is an isomorphism boils down to the equalities $(E_{\Delta}^{\sep})^{\Phi_{\Delta,p}}=(E^{\sep})^{\varphi=\hrm{Id}}=\hbb{F}$ and the $\hbb{F}$-linearity of all the maps in Definition \ref{maps}.
	\end{proof}
	
		\begin{coro}\label{inj_breuil_zab}
		For $\Pi \in \hrm{Rep}_{\hbb{F}}\, \gl{n}{\qp}$, we have a natural injection $\mathbf{V}_{\Delta,\xi}(\Pi)\hookrightarrow \mathbf{V}_{\xi}(\Pi)$.
	\end{coro}
	\begin{proof}
		Using \cite[Rem. 2 p. 13]{zabradi_gln} and the finiteness proved in \cite{zabradi_finiteness}, there is a natural surjection $$\mathbf{D}_{\xi}^{\vee}(\Pi) \twoheadrightarrow E \otimes_{E_{\Delta}} \mathbf{D}_{\Delta}^{\vee}(\Pi).$$ Let $\chi_{\hrm{GL}_n}$ be the $1$-dimensional $(\varphi,\Gamma)$-module such that $\mathbb{V}(\chi_{\hrm{GL}_n})=\delta_{\hrm{GL}_n}$ (the latter is defined in \cite[p. 23]{breuil2021conjectures}). It coincides with  $E \otimes_{E_{\Delta}} \chi^{\boxtimes}_{\Delta}$ where $\chi^{\boxtimes}_{\Delta}$ is the multivariable $(\varphi,\Gamma)$-module such that $\mathbb{V}_{\Delta}(\chi^{\boxtimes}_{\Delta})=\delta^{\boxtimes}_{\Delta}$ and the tensor product is along the canonical map $E_{\Delta}\rightarrow E$. We pass to duals and twist by $\chi_{\hrm{GL}_n}$, obtaining $$\und{\hrm{Hom}}(E\otimes_{E_{\Delta}} \mathbf{D}_{\Delta}^{\vee}(\Pi),\chi_{\hrm{GL}_n}) \hookrightarrow \und{\hrm{Hom}}(\mathbf{D}_{\xi}^{\vee}, \chi_{\hrm{GL}_n}).$$ 
		
		By \cite[Prop. 3.3 (3)]{marquis_formalisme}, internal $\hrm{Hom}$ and base change commute, which identifies the left term with \linebreak$E\otimes_{E_{\Delta}}\und{\hrm{Hom}}( \mathbf{D}_{\Delta}^{\vee}(\Pi),\chi^{\boxtimes}_{\Delta}) \simeq E\otimes_{E_{\Delta}} \mathbf{V}_{\Delta}(\Pi)$. We apply Fontaine's functor $\hbb{V}$ to the previous injection, which still produces an injection thanks to $\hbb{V}$'s exactness. The left term now identifies with $\mathbf{V}_{\Delta,\xi}(\Pi)$ by Lemma \ref{res_identif}. The right term now identifies with $\mathbf{V}_{\xi}(\Pi)$ thanks to the commutation of $\hbb{V}$ to internal Hom.
	\end{proof}

	\vspace{0.75 cm}
	
	\section{Imprecise study of $\overline{L}^{\boxtimes}(\rho)$ and first failure}

	In this section, we highlight a first failure of $\mathbf{V}_{\Delta}$ relying on two ingredients: \cite[Th. C]{zabradi_finiteness} and counting Jordan-Hölder components.
	
	\vspace{0.3 cm}
	
	In this section, we fix an $n$-dimensional object $\rho$ of $\hrm{Rep}_{\hbb{F}} \,\cg_{\qp}$. We fix a composition series for $\rho$ and write the successive irreducible pieces $(\rho_l)_{1\leq l \leq k}$ of respective dimensions $\und{n}=(n_l)$. 
	
	\begin{defi}
		For such a tuple $\und{n}$ and $i\geq 1$, define the set of cuts $$\hrm{Cut}_{i,\und{n}}:=\Big\{ (j_l)\in \N_{\geq 1}^{\llbracket 1,k\rrbracket} \, \Big|\,\substack{ \sum j_l=i \\ \forall l, \,\, 1\leq j_l \leq n_l} \Big\}.$$
	\end{defi}
	
	\begin{lemma}
		Let $1\leq i <n$. The representation $\Lambda^i \rho$ has more than $\# \hrm{Cut}_{i,\und{n}}$ Jordan-Hölder components.
	\end{lemma}
	\begin{proof}
		As our composition series for $\rho$ produces one for $\Lambda^i \rho$, the claim is equivalent to its analogue for $\rho^{\sms}=\oplus_l\, \rho_l$. In this case, we have an isomorphism $$\bigoplus_{(j_l)\in \hrm{Cut}_{i,\und{n}}} \bigotimes_{1\leq l \leq k} \left(\Lambda^{j_l} \rho_l\right) \rightarrow \Lambda^i \rho^{\sms}$$ with only nonzero terms in the sum. It concludes.
	\end{proof}
	
	\begin{coro}\label{counting_0}
		The representation $\overline{L}^{\boxtimes}(\rho)$ has more than $\prod_{1\leq i <n} \# \hrm{Cut}_{i,\und{n}}$ Jordan-Hölder components.
	\end{coro}
	\begin{proof}
		Combine the previous lemma with the following general fact: if $V$ and $W$ are representations of $G$ and $H$ with two normal series $(V_a)_{a\leq A}$ and $(W_b)_{b\leq B}$, then $(V_{a-1} \boxtimes W + V_a \boxtimes W_b)_{a,b\leq A,B}$ with lexicographic order is a normal serie for $V \boxtimes W$.
	\end{proof}

	 Until the end of the section, suppose that $\rho$ is generic and a good conjugate\footnote{This is very mild asumption: genericity takes away multiplicities and the pathological extensions of characters already appearing for $\gl{2}{\qp}$ (see \cite[VII.4 Atomes de longueur 3]{colmez_correspondance}). Under this hypothesis, \cite[\S 2.3.2]{breuil2021conjectures} shows that good conjugates exist and all have similar notions of compatibility.} in the sense of \cite[\S 2.3.2]{breuil2021conjectures}. With this assumption $P_{\rho}$ is the standard parabolic whose Levi is $M_{P_{\rho}}=\mathrm{GL}_{n_1}\times \ldots \times \mathrm{GL}_{n_k}$.
	 
	 \begin{lemma}\label{counting}
	 	The number of isotypic components of $\overline{L}^{\otimes}_{|Z_{M_{P_{\rho}}}}$ is smaller than $\prod_{1\leq i <n} \# \hrm{Cut}_{i,\und{n}}$, strictly if $\rho$ is not irreducible and $n\geq 3$.
	 \end{lemma}
	 \begin{proof}
	 	Call $(e_r)_{1\leq r\leq n}$ the canonical basis of the standard representation. The family $(\otimes_i (\wedge_{r\in I_i} e_r))_{(I_i)}$ for $(I_i)$ ranging over in $\prod_{1\leq i <n} \{ I\subset \llbracket 1,n\rrbracket \, |\, \# I=i\}$ forms a basis of $Z_{M_{P_{\rho}}}$-eigenvectors of $\overline{L}^{\otimes}$. To each eigenvector corresponds an algebraic character of $Z_{M_{\rho}} \simeq \hbb{G}_m^k$ i.e. a tuple of integers, which are positive for the standard representation. The isotypic components are hence parametrised by $$J_{\und{n}}:=\bigg\{ (\mu_l) \in \N^{\llbracket 1,k\rrbracket}\, \bigg|\,  \substack{\exists (I_i)_{1\leq i <n}  \, \text{subsets \, of \,} \llbracket 1,n\rrbracket \\ \text{s.t.} \, \forall i, \, \# I_i=i \\ \text{and}\, \forall l, \, \sum_i \#(I_i \cap \llbracket \sum_{l'<l} n_{l'} +1, \sum_{l'\leq l} n_{l'}\rrbracket )=\mu_l}\bigg\}.$$ We will say that $(I_i)$ is a witness of $(\mu_l)\in J_{\und{n}}$ if they satisfy the condition defining $J_{\und{n}}$.
	 	
	 	Consider the map $$\prod_{1\leq i <n} \hrm{Cut}_{i,\und{n}} \rightarrow J_{\und{n}}, \,\,\, (j_{i,l})_{i,l}\mapsto \left(\sum_{i} j_{i,l}\right)_{l}.$$ It is well defined because $(I_i:=\sqcup_l \llbracket \sum_{l'<l} n_{l'}+1 , \sum_{l'<l} n_{l'}+ j_{i,l})$ is a witness of the image of $(j_{i,l})$. It is also surjective, because any $(\mu_l)$ with witness $(I_i)$ is the image of $(\#(I_i \cap \llbracket \sum_{l'<l} n_{l'} +1, \sum_{l'\leq l} n_{l'}\rrbracket))_{i,l}$.
	 	
	 	This finishes to prove that the number of isotypic components is smaller than the cardinal of tuples of cuts in $\prod \, \hrm{Cut}_{i,\und{n}}$. Proving that it is strictly smaller amounts to finding a defect of injectivity of the map above when $\rho$ is reducible (i.e. $k\geq 2$) and $n\geq 3$.
	 	
	 	If $k\geq 3$, we get inspiration from the example of $n=3$ and $k=3$, i.e. $P_{\rho}=B\subset \hrm{GL}_3$. For this example, $e_1 \otimes (e_2 \wedge e_3)$ and $e_2 \otimes (e_1 \wedge e_3)$ are two eigenvectors for the same character \linebreak$\hrm{Diag}(x,y,z)\mapsto xyz$ of the torus, but they correspond to two pairs of cuts $((1,0,0),(0,1,1))$ and $((0,1,0),(1,0,1))$ that are witnesses of (the tuple associated with) the character. We mimic this in the general situation. Take any family $(I_i)_{3\leq i <n}$ and look at $(\mu_l)$ corresponding to $(\{1\},\{2,3\},I_3,\ldots,I_{n-1})$. Then $(\{2\},\{1,3\},I_3,\ldots,I_{n-1})$ also witnesses the belonging of $(\mu_l)$ to $J_{\und{n}}$. These two witnesses give two elements in the preimage of $(\mu_l)$.
	 	
	 	If $k=2$, there exists an $l_0$ such that $n_{l_0}\geq 2$. We get inspiration from the example of $n=3$ and $(n_1,n_2)=(1,2)$, i.e. $P_{\rho}$ is the parabolic associated to $\mathbb{G}_m \times \hrm{GL}_2$. In this situation, the same two vectors as before are eigenvectors for the character $\hrm{Diag}(x,y,y)\mapsto xy^2$ of $Z_{M_{P_{\rho}}}$ but produce different witnesses of (the tuple associated with) the character. We mimic this in the general situation fixing $l_0$ as above and $l_1\neq l_0$. Take any family $(I_i)_{3\leq i <n}$ and look at $(\mu_l)$ corresponding to \linebreak$(\{\sum_{l<l_0} n_l +1\},\{\sum_{l<l_1} n_l +1 ,\sum_{l<l_0} n_l +2\},I_3,\ldots,I_{n-1})$. In this situation, the family \linebreak$(\{\sum_{l<l_1} n_l +1\},\{\sum_{l<l_0} n_l +1,\sum_{l<l_0} n_l +2\},I_3,\ldots,I_{n-1})$ also witnesses the belonging of $(\mu_l)$ to $J_{\und{n}}$. Same conclusion.
	 \end{proof}
	
	\begin{defi}		
		We say that $\Pi\in \hrm{Rep}_{\hbb{F}} \, \gl{n}{\qp}$ is compatible with $\widetilde{P}_{\rho}$ if it comes from $\Pi'$ defined over $\hbb{F}_q$ and compatible with $\widetilde{P}_{\rho}$ for all sufficiently big subfields $\hbb{F}_q\subset \hbb{F}$.
	\end{defi}
	
	\begin{rem}
		For $\hbb{F}$ finite, this amounts to being compatible with $\widetilde{P}_{\rho}$. For $\hbb{F}$ infinite, we expect that imposing such compatibility for all $q$ big enough rather than for one should not be a restriction for the representations we study nowaday, that is the global candidates that should realise a hypothetical local Langlands correspondance, e.g. the representations $S(U^p,k_E)[\mathfrak{m}]$ in \cite{breuil_herzig_ordinary}. We take as hints the fact that for $\gl{2}{K}$, the candidates considered in \cite{bhhms4} are base change of one another when we choose bigger and bigger $\hbb{F}$ and that the results about their structure do not depend on the chosen (big enough) $\hbb{F}$. Another hint comes from the construction of $\Pi(\rho)^{\hrm{ord}}$ in \cite{breuil_herzig_ordinary} for generic ordinary $n$-dimensional $\rho$, which doesn't depend on the chosen finite extension up to base change\footnote{Use \cite[Lem. 3.4.1, Prop. 3.3.3]{breuil_herzig_ordinary} to see successively that $C_{\rho}$ and $\widetilde{\Pi}(\rho)_I$ don't depend on the chosen finite extension up to base change.}.
	\end{rem}
	
	 \begin{theo}\label{main_failure_1}
	 	Let $n\geq 3$ and $\rho$ be an $n$-dimensional reducible object of $\hrm{Rep}_{\hbb{F}} \, \cg_{\qp}$, which is generic and a good conjugate. Then any  $\Pi$ in $\hrm{Rep}_{\hbb{F}} \, \gl{n}{\qp}$ that is compatible with $\widetilde{P}_{\rho}$ satisfies $$\mathbf{V}_{\Delta}(\Pi) \not\simeq \overline{L}^{\boxtimes}(\rho).$$
	 \end{theo}
	 \begin{proof}
	By design, we can suppose that $\hbb{F}$ is finite.
	
	We first investigate the image of a Jordan-Hölder component of $\Pi$ by $\mathbf{V}_{\Delta}$. For each one of these, the conditions (i) and (iv) of \cite[Def. 2.4.1.5]{breuil2021conjectures} tell that there exists a standard parabolic $Q$ with Levi $M_Q=\hrm{GL}_{m_1}\times \ldots \times \hrm{GL}_{m_s}$ and a supercuspidal representation $\boxtimes_t\,  \Pi_t$ of $M_Q$ such that the component is isomorphic to $\hrm{Ind}_{Q^-(\qp)}^{\hrm{GL}_n(\qp)} (\boxtimes_t\,  \Pi_t)$. As usual $\boxtimes_t\, \Pi_t$ is seen as a $Q^-(\qp)$-representation via inflation from $M_Q$. Note that this is equivalent to each $\Pi_t$ being a supercuspidal representation of $\gl{m_t}{\qp}$. Thanks to Theorem \ref{compat_induction}, we have $$\mathbf{D}_{\Delta}^{\vee}\Big(\hrm{Ind}_{Q^-(\qp)}^{\hrm{GL}_n(\qp)} (\boxtimes_t\,  \Pi_t)\Big)\simeq \bigotimes_{1\leq t\leq s} \big( E_{\Delta}\otimes_{E_{\Delta_t}}\mathbf{D}_{\Delta_t}^{\vee}(\Pi_t)\big)$$ where $\Delta_t$ is set standard set of simple roots of $\hrm{GL}_{m_t}$. We then use Proposition \ref{vdelta_sqcup} to see that $$(\mathbb{V}_{\Delta}\circ\mathbf{D}_{\Delta}^{\vee})\Big(\hrm{Ind}_{Q^-(\qp)}^{\gl{n}{\qp}} (\boxtimes_t \,  \Pi_t)\Big) \simeq \mathop{\boxtimes} \limits_{1\leq t\leq s} (\mathbb{V}_{\Delta_t}\circ \mathbf{D}_{\Delta_t}^{\vee})(\Pi_t).$$ In our situation, tensor product and dual commutes; the representation $\mathbf{V}_{\Delta}(\hrm{Ind}_{Q^-(\qp)}^{\hrm{GL}_n(\qp)} (\boxtimes \Pi_t))$ is the inflation $\hrm{Inf}_{\prod_t \cg_{\qp,\Delta_t}}^{\cg_{\qp,\Delta}}\big(\boxtimes_t \mathbf{V}_{\Delta_t}(\Pi_t)\big)$ along the projection $\cg_{\qp,\Delta}\rightarrow \prod \cg_{\qp,\Delta_t}$, twisted by a character. Each $\mathbf{V}_{\Delta_t}(\Pi_t)$ is zero or irreducible thanks to \cite[Th. C]{zabradi_finiteness}. Hence, so is the inflation of their boxproduct. We showed that the image by $\mathbf{V}_{\Delta}$ of each Jordan-Hölder factor of $\Pi$ is either zero or irreducible.
	
	Thanks to \cite[Lem. 2.4.1.1 (i)]{breuil2021conjectures}, the number of Jordan-Hölder components of $\Pi$ is precisely the number of isotypic components of $\overline{L}^{\otimes}_{|Z_{M_{P_{\rho}}}}$ i.e. $J_{\und{n}}$.  Using right exactness of $\mathbf{V}_{\Delta}$, the number of Jordan-Hölder components of $\mathbf{V}_{\Delta}(\Pi)$ is less than $\# J_{\und{n}}$.	Thanks to Corollay \ref{counting_0} and Lemma \ref{counting}, this is less than the number of irreducible components of $\overline{L}^{\boxtimes}(\rho)$. The two Galois representations are not isomorphic.
	 \end{proof}

	 The same counting of Jordan-Hölder components can work for irreducible representations. 
	 
	 \begin{theo}\label{irred_fail}
	 	There exist infinitely many integers $n$ such that there is an $n$-dimensional irreducible $\rho\in \hrm{Rep}_{\hbb{F}} \,\cg_{\qp}$, with no supercuspidal $\Pi\in \hrm{Rep}_{\hbb{F}} \,\gl{n}{\qp}$ verifying $\mathbf{V}_{\Delta}(\Pi)\simeq\overline{L}^{\boxtimes}(\rho)$.
	 \end{theo}
	 \begin{proof}
		 By Theorem \ref{irred_or_zero}, any $\mathbf{V}_{\Delta}(\Pi)$ for $\Pi$ supercuspidal is either zero or irreducible. As soon as $\Lambda^i \rho$ is reducible for some $i\geq 2$, then $\overline{L}^{\boxtimes}(\rho)$ is reducible; we construct a family of representations $\rho$ with unbounded dimensions for which $\Lambda^2 \rho$ is reducible.
		 
		Recall from local class field theory that $$ \cg_{\qp}/\hcal{I}_{\qp}^{\hrm{wild}} \simeq \left(\lim \limits_{\substack{\longleftarrow \\r\geq 1, \, \hrm{norm \, map}}} \mathbb{F}_{p^r}^{\times} \right) \rtimes \varphi^{\widehat{\Z}}.$$ Take any prime $\ell$ with $\ell \nmid p(p^2-1) (p^3-1)$ and $p\nmid (\ell -1)$. They are infinitely many such by Dirichlet's prime number theorem. As $\ell |(p^r-1)$ for some $r$, we can identify $\Z/\ell \Z$ with a quotient of $\lim \limits_{\leftarrow} \hbb{F}_{p^r}^{\times}$ by a characteristic subgroup $H$. The conjugation action of $\varphi^{\widehat{\Z}}$ on this quotient identifies with the multiplication by some element in $(\Z/\ell \Z)^{\times}$. Write the kernel of this action as $\varphi^{m\widehat{\Z}}$. Consider the quotient $Q_{\ell}$ of $\cg_{\qp}/\hcal{I}_{\qp}^{\hrm{wild}}$ by $H \rtimes \varphi^{m\widehat{\Z}}$ that we identify with $(\Z/\ell\Z) \rtimes K_{\ell}$ for some $K_{\ell} \triangleleft (\Z/\ell \Z)^{\times}$. Remark that $\# K_{\ell}$ is the order of $p$ in $(\Z/\ell \Z)^{\times}$. For any $d$, the primes $\ell$ stop dividing $p^d-1$ at some point, hence $\# K_{\ell} \xrightarrow{\ell \rightarrow +\infty} +\infty$.
		
		We now look more precisely at the irreducible representations of $Q_{\ell}$ over $\overline{\mathbb{F}}_p$. The conjugacy classes of $Q_{\ell}$ are the sets $\{\hrm{Id}\}$, $a K_{\ell} \times \{1\}$ for $a \in (\Z/\ell\Z)\backslash\{0\}$ and $ (\Z/\ell\Z) \times\{x\}$ for $x\in K_{\ell}\backslash \{1\}$: there are $\# K_{\ell} + (l-1)/\# K_{\ell}$ of them. We also have $\# K_{\ell}$ irreducible non isomorphic $1$-dimensional representations, obtained by inflation from $K_{\ell}$. Our hypotheses on $\ell$ imply that $p\nmid \# Q_{\ell}$; Frobenius theorem gives $$\# K_{\ell} + (l-1) m^2/(\# K_{\ell}) \geq \sum_{\substack{[V] \, \text{isomorphism classes of} \\ \text{irreducible representations of }Q_{\ell}\text{ over } \ov{\mathbb{F}}_p}} (\dim V)^2= \# Q_{\ell}=\ell (\# K_{\ell})$$ where $m$ is the maximum dimension of such irreducible representation. In particular, we have $m\geq \# K_{\ell}$. Pick one irreducible of dimension $m$. The (finite) sum of its distinct conjugates by $\cg_{\mathbb{F}}$ descends to an irreducible representation over $\mathbb{F}$. It must be of maximal dimension among such irreducibles and this dimension is greater than $\# K_{\ell}$. We fix such a representation $\sigma_{\ell}$. Because $\ell \nmid (p^2-1) (p^3-1)$, we have $\#K_{\ell} \geq 4$. As $\dim \Lambda^2 \sigma_{\ell} =\dim \sigma_{\ell} (\dim \sigma_{\ell} -1)/2 > \dim \sigma_l$, the representation $\bigwedge^2 \sigma_{\ell}$ cannot be irreducible. 
		
		Adding everything, the inflations $\hrm{Inf}_{Q_{\ell}}^{\cg_{\qp}} \sigma_{\ell}$ form a family of irreducible representations with unbounded dimensions and reducible $\Lambda^2$.
	 \end{proof}

	 \vspace{0.75 cm}
	\section{Setup for the $\hrm{GL}_3$-toy case}

	Our toy example for studying a "weak compatibility" hypothesis is a $3$-dimensional representation $\rho$ of $\cg_{\qp}$ of length $3$, which we suppose to be generic and good conjugate. Here, it means that $\rho$ can be written $$\begin{psmallmatrix} \chi_1 & \delta_a & \epsilon \\ & \chi_2 & \delta_b \\ & & \chi_3 \end{psmallmatrix}$$ with all $\chi_i \chi_j^{-1} \notin \{1,\omega^{\pm 1}\}$. 
	
	The first thing we do is analyse the shape of $\overline{L}^{\boxtimes}(\rho)$ and $\overline{L}^{\otimes}(\rho)$. Recall a slight variant of the notion of module diagram introduced in \cite{ALPERIN1980111}.
	
	\begin{defi}[Variant of {\cite[\S 3-4]{ALPERIN1980111}}]
		A module diagram is the finite directed graph $\mathcal{D}$ obtained from a poset $X$ by taking $X$ as set of vertices and adjoining one edge from $x$ to $y$ if $x<y$ and if there does not exist $x<z<y$.
		
		For a poset $A$, we define the joint (resp. intersection) of two elements $a$ and $b$, when it exists, as the least upper bound (resp. the greatest lower bound) of $\{a,b\}$. For an element $a$, we consider the set of maximal elements of $\{b<a\}$ : their greatest lower bound is called the radical of $a$. 
		
		For a module diagram $\mathcal{D}$, we say that a subgraph  $\hcal{D}'\subset \hcal{D}$  is initial if \begin{tikzcd}[column sep=scriptsize] x &  y \arrow["{>}"{marking, allow upside down}, no head, from=1-1, to=1-2]\end{tikzcd} and $y\in \hcal{D}'$ imply $x\in \hcal{D}'$. Initial subgraphs form a poset with inclusion as order. Joints (resp. intersections, resp. radicals) exist and correspond to unions of subgraphs (resp. intersection of subgraphs, resp. intersection of maximal strict subgraphs).
		
		Let $V$ be a finite length representation of a group $G$ over a field $k$. A labelled module diagram $\mathcal{D}$ for $V$ is a pair $(\mathcal{D}, \delta)$ where $\mathcal{D}$ is a module diagram with $\hrm{length} \,V$ vertices and $\delta$ is a function from the poset of initial subgraphs of $\hcal{D}$ to subrepresentations of $V$ which satisfies $\delta(\hcal{D})=V$ and preserves joints, intersections and radicals.
		
		Let $(\hcal{D},\delta)$ be a labelled module diagram for $V$. 	For each vertex $x$, define an initial subgraph $\hcal{D}_{\leq x}:=\{y\in \hcal{D}\, |\, y\leq x\}$. Then, \cite[Th. 3]{ALPERIN1980111} proves that $x \mapsto \delta(\hcal{D}_{\leq x})/\delta(\hrm{rad} \, \hcal{D}_{\leq x})$ establishes a bijection between the vertices of $\mathcal{D}$ and the set $\hrm{JH}(V)$ of Jordan-Hölder components of $V$ counted with multiplicity. We will often synthetise a labelled module diagram by writing only the directed graph with vertices labelled by $\hrm{JH}(V)$ (which contains slightly less information). 
	\end{defi}
	
	\begin{rem}
		For any total ordering $\preceq$ on $\hrm{JH}(V)$ extending the graph order, evaluating $\delta$ at \linebreak$\mathcal{D}_{\preceq x} :=\{y\preceq x\}$ provides a composition series for $V$ with gradings $\hrm{JH}(V)$ in the order $\preceq$.
		
		There can be several labelled module diagrams for $V$. As basic examples, if $V$ is a direct sum of irreducibles $V=\oplus_{i\leq n} V_i$, any partial order $\leq$ on $\{V_i\}$, gives a labelled module diagram with $$\delta(\mathcal{D}')=\bigoplus_{V_j \text{ vertex of } \mathcal{D}'} V_i.$$ Worse, if $V=W^{\oplus 2}$ with $W$ irreducible, any couple of embeddings $\iota_1, \iota_2 \, : \, W\hookrightarrow V$ with distinct images gives a labelled module diagram where $\mathcal{D}$ has vertices $\{1,2\}$, no edge, and where $$\delta(\mathcal{D}')=\sum_{j \text{ vertex of } \mathcal{D}'} \hrm{Im}(\iota_j).$$
		
		When $V$ is multiplicity free, there exists an initial labelled module diagram for $V$. Its directed graph is constructed in \cite[\S 2]{ALPERIN1980111}, and has underlying partial order on Jordan-Hölder factors defined by \linebreak$W_1\leq W_2$ iff $V(W_1)\subseteq V(W_2)$ where $V(W)$ is the unique subrepresentation with cosocle $W$.
	\end{rem}
	
	\begin{lemma}
		The representation $\Lambda^2 \rho$ has length $3$ and a labelled module diagram looking like \begin{center}\begin{tikzcd}
			{\chi_1 \chi_2} && {\chi_1 \chi_3} && {\chi_2 \chi_3}
			\arrow["{>}"{marking, allow upside down}, no head, from=1-1, to=1-3]
			\arrow["{>}"{marking, allow upside down}, no head, from=1-3, to=1-5]
		\end{tikzcd}\end{center}
		
		Its socle (resp. its cosocle) has the same dimension as the cosocle (resp. the socle) of $\rho$.
	\end{lemma}
	\begin{proof}
		In the basis $\{e_1\wedge e_2, \, e_1 \wedge e_3, \, e_2\wedge e_3\}$ the matrix representation is $$\begin{psmallmatrix} \chi_1 \chi_2 & \chi_1 \delta_b & \delta_a \delta_b - \epsilon \chi_2 \\  & \chi_1 \chi_3 & \delta_a \chi_3 \\  &  & \chi_2 \chi_3 \end{psmallmatrix}.$$
		
		Let us give one argument for socle/cosocle analysis, the other being similar. The subrepresentation $\hrm{Vect}(e_1\wedge e_2, \, e_1 \wedge e_3)$ is semisimple iff $\delta_b$ is trivial, hence iff  $\rho/\chi_1$ is semisimple. 
	\end{proof}
	
	\begin{prop}
		The representation $\ov{L}^{\boxtimes}(\rho)$ has a labelled module diagram looking like 
		
		\begin{center}
		\begin{tikzcd}
			&& {\chi_1 \boxtimes (\chi_2\chi_3)} && \\
			& {\chi_1 \boxtimes (\chi_1\chi_3)} && {\chi_2 \boxtimes (\chi_2\chi_3)} \\
			{\chi_1 \boxtimes (\chi_1\chi_2)} && {\chi_2 \boxtimes (\chi_1\chi_3)} && {\chi_3 \boxtimes (\chi_2\chi_3)} \\
			& {\chi_2 \boxtimes (\chi_1\chi_2)} && {\chi_3 \boxtimes (\chi_1\chi_3)} \\
			&& {\chi_3 \boxtimes (\chi_1\chi_2)}
			\arrow["{>}"{marking, allow upside down}, no head, from=1-3, to=2-4]
			\arrow["{>}"{marking, allow upside down}, no head, from=2-2, to=1-3]
			\arrow["{>}"{marking, allow upside down}, no head, from=2-2, to=3-3]
			\arrow["{>}"{marking, allow upside down}, no head, from=2-4, to=3-5]
			\arrow["{>}"{marking, allow upside down}, no head, from=3-1, to=2-2]
			\arrow["{>}"{marking, allow upside down}, no head, from=3-1, to=4-2]
			\arrow["{>}"{marking, allow upside down}, no head, from=3-3, to=2-4]
			\arrow["{>}"{marking, allow upside down}, no head, from=3-3, to=4-4]
			\arrow["{>}"{marking, allow upside down}, no head, from=4-2, to=3-3]
			\arrow["{>}"{marking, allow upside down}, no head, from=4-2, to=5-3]
			\arrow["{>}"{marking, allow upside down}, no head, from=4-4, to=3-5]
			\arrow["{>}"{marking, allow upside down}, no head, from=5-3, to=4-4]
		\end{tikzcd}
		\end{center}
		
		When $\rho$ has Loewy length $3$ (resp. $2$, resp. is semisimple),  $\ov{L}^{\boxtimes}(\rho)$ has Loewy length $5$ (resp. $3$, resp. $1$) with dimensions of the graded pieces $(1,2,3,2,1)$ (resp. $(2,5,2)$, resp. $9$).
	\end{prop}
	\begin{proof}
		Use the convolution formula for boxproducts of smooth representations of finite dimension. One could also write down explicitely the exterior tensor product of matrices.
	\end{proof}
	
	We write $\overline{L}^{\boxtimes}(\rho)^{\ord}$ for the subrepresentation of $\overline{L}^{\boxtimes}(\rho)$ having $(\chi_1 \boxtimes \chi_1 \chi_2)$, $(\chi_1 \boxtimes \chi_1 \chi_3)$ and \linebreak$(\chi_2 \boxtimes \chi_1 \chi_2)$ as Jordan-Hölder factors. In case $\rho$ is of Loewy length $3$ (equivalently  when it is maximally non split), it is $\hrm{soc}^2_{(\cg_{\qp}\times \cg_{\qp})} \overline{L}^{\boxtimes}(\rho)$. We exchange the symbol $\boxtimes$ for an $\otimes$ to denote their restriction to the diagonal embedding of $\cg_{\qp}$.
	
	\begin{rem}
	The representation $\ov{L}^{\otimes}(\rho)^{\hrm{ord}}$ coincides with $(\ov{L}^{\otimes}_{|B})^{\hrm{ord}} \circ \rho$ where the first is defined in \cite[\S 2.5]{breuil_herzig_ordinary}. See also \cite{breuil_ind_parab} where this composition appears.
	\end{rem}

	\begin{coro}\label{coro_subrep_restriction}
		If $p\neq 2$, the $\cg_{\qp}$-representation $\overline{L}^{\otimes}(\rho)/\overline{L}^{\otimes}(\rho)^{\hrm{ord}}$  has a labelled module diagram looking like
		\begin{center}
			\begin{tikzcd}
				{\mathrm{det}\,\rho} & {\chi_2^2\chi_3} & \\
				{\mathrm{det}\,\rho} && {\chi_2\chi_3^2} \\
				{\mathrm{det}\,\rho} & {\chi_1\chi_3^2}
				\arrow["{>}"{marking, allow upside down}, no head, from=1-1, to=1-2]
				\arrow["{>}"{marking, allow upside down}, no head, from=1-2, to=2-3]
				\arrow["{>}"{marking, allow upside down}, no head, from=2-1, to=2-3]
				\arrow["{>}"{marking, allow upside down}, no head, from=3-1, to=3-2]
				\arrow["{>}"{marking, allow upside down}, no head, from=3-2, to=2-3]
			\end{tikzcd}
		\end{center} when $\rho$ is of Loewy length $3$, the columns give the socle filtration.
	\end{coro}
	\begin{proof}
		In the basis $\{e_1 \otimes (e_2 \wedge e_3), e_2 \otimes (e_1\wedge e_3), e_3 \otimes (e_1 \wedge e_2), e_2 \otimes (e_2 \wedge e_3), e_3 \otimes (e_1 \wedge e_3), e_3\otimes (e_2\wedge e_3)\}$, the matrix looks like $$\begin{pmatrix} \det \rho & 0 & 0 & \chi_2\chi_3 \delta_a & 0 & * \\ & \det \rho & 0 & \chi_2 \chi_3 \delta_a & \chi_1 \chi_3 \delta_b & * \\ & & \det \rho & 0 & \chi_1 \chi_3 \delta_b & * \\ & & & \chi_2^2 \chi_3 & 0 & \chi_2 \chi_3 \delta_b \\ & & & & \chi_1 \chi_3^2 & \chi_3^2 \delta_a \\ & & & & & \chi_2 \chi_3^2 \end{pmatrix}$$ We apply the base change  by the matrix $\begin{psmallmatrix} 1 & 1 & 0 & 0 \\ 0 & -1 & 0 \\ 0 & 1 & 1 & 0 \\ 0& 0 &0 & \hrm{Id}_2 \end{psmallmatrix}$, which is invertible as soon as $p\neq 2$. We obtain $$\begin{pmatrix} \det \rho & 0 & 0 & \chi_2\chi_3 \delta_a & 0 & * \\ & \det \rho & 0 & 0 & 0 & * \\ & & \det \rho & 0 & \chi_1 \chi_3 \delta_b & * \\ & & & \chi_2^2 \chi_3 & 0 & \chi_2 \chi_3 \delta_b \\ & & & & \chi_1 \chi_3^2 & \chi_3^2 \delta_a \\ & & & & & \chi_2 \chi_3^2 \end{pmatrix}$$ which gives the announced results.
		
	\end{proof}

	\begin{rem}
		Be careful that socle filtration and restriction to the diagonal embedding do not commute: for instance, there is always a copy of $(\det \rho)$ in $\hrm{soc}^2_{\cg_{\qp}}(\overline{L}^{\otimes}(\rho))$.
	\end{rem}

\vspace{0.5 cm}

For $\rho$ of length $3$ and $\hbb{F}$ finite, the representation predicted by \cite[\S 2.4.3 Ex. 2]{breuil2021conjectures} has labelled module diagram looking like 

\begin{equation}\label{eq:form}
	\begin{tikzcd}
		& {\mathrm{PS}_{\chi_2,\chi_1,\chi_3}} && {\mathrm{PS}_{\chi_2,\chi_3,\chi_1}} & \\
		{\mathrm{PS}_{\chi_1,\chi_2,\chi_3}} && {\mathrm{SS}} && {\mathrm{PS}_{\chi_3,\chi_2,\chi_1}} \\
		& {\mathrm{PS}_{\chi_1,\chi_3,\chi_2}} && {\mathrm{PS}_{\chi_3,\chi_1,\chi_2}}
		\arrow["{>}"{marking, allow upside down}, no head, from=1-2, to=2-3]
		\arrow["{>}"{marking, allow upside down}, no head, from=1-4, to=2-5]
		\arrow["{>}"{marking, allow upside down}, no head, from=2-1, to=1-2]
		\arrow["{>}"{marking, allow upside down}, no head, from=2-1, to=3-2]
		\arrow["{>}"{marking, allow upside down}, no head, from=2-3, to=1-4]
		\arrow["{>}"{marking, allow upside down}, no head, from=2-3, to=3-4]
		\arrow["{>}"{marking, allow upside down}, no head, from=3-2, to=2-3]
		\arrow["{>}"{marking, allow upside down}, no head, from=3-4, to=2-5]
	\end{tikzcd}
\end{equation} with $\theta$ being the product of simple roots, the representations $\hrm{PS}_{\chi'}$ defined as $\hrm{Ind}_{B^{-}(\qp)}^{\hrm{GL}_3(\qp)} \chi' (\varepsilon^{-1} \circ \theta)$. If $\rho$ has Loewey length $3$, we have $P_{\rho}=\widetilde{P}_{\rho}=B$ and the module diagram above is the initial one, i.e. all extensions are nonsplit.

Moreover, any representation $\Pi$ over a finite $\hbb{F}$ compatible with $B$ must have an initial labelled module diagram as above for some triple $(\chi_1,\chi_2,\chi_3)$. This is a translation of \cite[Def. 3.4.1.5]{breuil2021conjectures}.

\begin{defi}
	In general, for $\Pi\in \hrm{Rep}_{\hbb{F}} \, \gl{3}{\qp}$ compatible with $B$, coming from some $\Pi'$ in \linebreak$\hrm{Rep}_{\hbb{F}_q} \, \gl{3}{\qp}$ compatible with $B$, we note $\Pi^{\ord}$ (resp. $\Pi_{\hrm{SS}}$, resp. $\Pi_{\ord}$) the base change of the subquotients $\mathrm{soc}^2_{\gl{3}{\qp}}\Pi'$ (resp. $\hrm{soc}_{3,\gl{3}{\qp}}\Pi'$, resp. $\Pi/\hrm{soc}^3_{\gl{3}{\qp}}\Pi'$).
	
	As any two such $\Pi'$ become equal after base change to a sufficiently large extension, the definition doesn't depend of the choice of $\Pi'$.
\end{defi}

\begin{rem}
	In \cite{breuil_herzig_ordinary}, Breuil and Herzig constructed a representation $\Pi(\rho)^{\ord}$ that is conjectured to be the maximal subrepresentation of some $\Pi(\rho)$ constructed by global methods, with only principal series as Jordan-Hölder components. It coincides with $\Pi^{\ord}$ in our setting for any $\Pi$ compatible with $B$ and triple $(\chi_1,\chi_2,\chi_3)$ being the ordered Jordan-Hölder components of $\rho$.
\end{rem}

We first explain why Z\'abr\'adi's functor applied to $\Pi^{\ord}$ recovers $\overline{L}^{\boxtimes}(\rho)^{\ord}$.

\begin{defi}
	For a smooth character $\chi \, : \, T(\qp) \rightarrow \overline{\mathbb{F}}_p^{\times}$, we view it as a character of the product \linebreak$(\hrm{W}_{\qp}^{\hrm{ab}}\times \hrm{W}_{\qp}^{\hrm{ab}} \times \qptimes)$ via $$(\hrm{W}_{\qp}^{\hrm{ab}}\times \hrm{W}_{\qp}^{\hrm{ab}} \times \qptimes) \xrightarrow{\sim} T(\qp), \,\,\, (g,h,x)\mapsto \hrm{Diag}(\hrm{Art}^{-1}(gh) x, \, \hrm{Art}^{-1}(h)x,x).$$ We restrict to $(\hrm{W}_{\qp}^{\hrm{ab}}\times \hrm{W}_{\qp}^{\hrm{ab}})$, extend to the profinite completion and inflate, obtaining a character $\chi_{\Delta}$ of $(\cg_{\qp}\times \cg_{\qp})$.
\end{defi}

\begin{lemma}\label{calcul_ind}
	Let $\chi \, : \, T(\qp) \rightarrow \overline{\mathbb{F}}_p^{\times}$ be a smooth character, seen by inflation as a character of $B^-(\qp)$. Then, $$\mathbf{V}_{\Delta}\left(\hrm{Ind}_{B^-(\qp)}^{\hrm{GL}_3(\qp)}\chi\right)=\chi_{\Delta}\otimes  \delta^{\boxtimes}_{\Delta}.$$
\end{lemma}
\begin{proof}
	For $G=T$, the set $\Delta_G$ is empty, the monoid $T_+$ of Z\'abr\'adi is $T(\qp)$ and the functor $\mathbf{D}_{\emptyset}^{\vee}$ sends a character on its dual seen as zero variable $T_+$-module. The result follows from Theorem \ref{compat_induction} and our expression of $\mathbb{V}_{\Delta}^*$.
\end{proof}

\begin{prop}\label{bij_PS}
	The functor $\mathbf{V}_{\Delta}$ establishes an isomorphism between $$\hrm{Ext}^1_{\gl{3}{\qp}}\left(\hrm{Ind}_{B^{-}(\qp)}^{\hrm{GL}_3(\qp)} s_{\alpha}(\chi) (\varepsilon^{-1}\circ \theta) \oplus \hrm{Ind}_{B^{-}(\qp)}^{\hrm{GL}_3(\qp)} s_{\beta}(\chi) (\varepsilon^{-1}\circ \theta),  \, \hrm{Ind}_{B^{-}(\qp)}^{\hrm{GL}_3(\qp)}\chi (\varepsilon^{-1}\circ \theta)\right)$$ where the extensions are taken in $\hrm{Rep}_{\hbb{F}} \,\gl{3}{\qp}$ and $$\hrm{Ext}^1_{(\cg_{\qp}\times \cg_{\qp})}\big((\chi_2\boxtimes (\chi_1\chi_2)) \oplus (\chi_1\boxtimes (\chi_1\chi_3) ), (\chi_1 \boxtimes (\chi_1\chi_2))\big).$$ 

The functor $\mathbf{V}_{\Delta,\xi}$ establishes a bijection between the same $\gl{3}{\qp}$-Ext-group and $$\hrm{Ext}^1_{\cg_{\qp}}(\chi_1\chi_2^2 \oplus \chi_1^2 \chi_3, \chi_1^2 \chi_2).$$

	The same kind of result holds for $$\hrm{Ext}^1_{\gl{3}{\qp}} \left(\hrm{Ind}_{B^{-}(\qp)}^{\hrm{GL}_3(\qp)} \chi^s (\varepsilon^{-1}\circ \theta), \, \hrm{Ind}_{B^{-}(\qp)}^{\hrm{GL}_3(\qp)}s_{\beta}(\chi^s) (\varepsilon^{-1}\circ \theta) \oplus \hrm{Ind}_{B^{-}(\qp)}^{\hrm{GL}_3(\qp)} s_{\alpha}(\chi^s) (\varepsilon^{-1}\circ \theta)\right)$$ with $\chi^s$ being the conjugate by the element of maximal length in the Weyl group of $\hrm{GL}_3$.
\end{prop}
\begin{proof}
	The previous lemma implies that $$\mathbf{V}_{\Delta}(\hrm{Ind}_{B^{-}(\qp)}^{\hrm{GL}_3(\qp)}\chi (\varepsilon^{-1}\circ \theta))=(\chi (\varepsilon^{-1}\circ \theta))_{\Delta} \otimes \delta^{\boxtimes}_{\Delta}= \chi_1 \boxtimes (\chi_1 \chi_2)$$ and similarily for the others. Moreover, the genericity hypothesis and \cite[Th. 4]{OLLIVIER200639} prove that all these inductions lie in $\hrm{SP}^0$. With this remark, it is a reformulation of \cite[Cor. 3.3.4]{zabradi_finiteness}. 
	
	Let's give the main ideas. First, the map is injective thanks to the faithfulness in Theorem \ref{exact_PS}. For essential surjectivity, we can do it independantly for the two extensions and suppose $\hbb{F}$ finite. Pick the first extension. Thanks to \cite[Prop. 3.3.2]{zabradi_finiteness}, this extension must be trivial restricted to $\{\hrm{Id}\} \times \cg_{\qp}$. Genericity tell that the extension group is of dimension $1$. Therefore \cite[Th. 1.1 (ii)]{Hauseux_2016} produces a representation induced from the $\qp$-points of the parabolic $P^-_{\alpha}$ with Levi $\hrm{GL}_2 \times \mathbb{G}_m$ which has the correct image.
\end{proof}

\begin{rem}
	It is an improvement, using Z\'abr\'adi's functor, of \cite[Th. 1.1]{breuil_ind_parab} for $G=\hrm{GL}_3$.
\end{rem}

	\vspace{0.75 cm}
	\section{Finer failure on our toy example}
	
		We keep our representation $\rho$ of dimension $3$ and Loewy length $3$.
		
		\begin{defi}
			We say that $\Pi$ is weakly compatible with $\rho$ if it is compatible with $\widetilde{P}_{\rho}$ and\footnote{This is condition (i) in \cite[Def. 2.4.2.7]{breuil2021conjectures} but only for $\Pi'=\Pi$.} if \linebreak$\mathbf{V}_{\xi}(\Pi)\simeq \overline{L}^{\otimes}(\rho)$.
		\end{defi}

		\begin{lemma}\label{mult_one}
			Let $G$ be a profinite group and $k$ an algebraically closed field. Let $V$ and $W$ be irreducible objects in $\hrm{Rep}_k \, G$ and $\chi$ a $1$-dimensional character. Then $\chi^{\oplus 2}$ can't be a subrepresentation of $V\otimes W$.
		\end{lemma}
		\begin{proof}
			As everything is finite dimensional, we have
			\begin{align*}
				\hrm{Hom}_{k[G]}(\chi, V\otimes W) &\simeq \hrm{Hom}_{k[G]}\left((V\otimes W)^{\vee}, \chi^{-1}\right) \\
				&\simeq \hrm{Hom}_{k[G]}\left(V^{\vee} \otimes W^{\vee},\chi^{-1}\right) \\
				& \simeq \hrm{Hom}_{k[G]}\left(V^{\vee}, W\otimes \chi^{-1}\right)
			\end{align*} As $V$ is irreducible, so is $V^{\vee}$. As $W$ is irreducible, so is $W\otimes \chi^{-1}$. Hence, Schur's lemma says that this $\hrm{Hom}_{k[G]}$ is at most of $1$-dimensional. It concludes.
		\end{proof}
		
		\begin{lemma}\label{triv_inertia_extension}
			Let $V$ belong to $\hrm{Rep}_{\ov{\mathbb{F}}_p}\, \cg_{\qp}$ such that $\hcal{I}_{\qp}^{\hrm{wild}}$ acts trivially. Any subrepresentation of $V$ that is an extension of two distinct characters is semisimple.
		\end{lemma}
		\begin{proof}
			Fix such subrepresentation $W$. Since any character is trivial when restricted to the pro-$p$-group $\hcal{I}_{\qp}^{\hrm{wild}}$, one can twist and suppose that $W$ is an extension of $1$ by $\chi$, i.e. lives in $H^1(\cg_{\qp}/\hcal{I}_{\qp}^{\hrm{wild}}, \chi)$. We reduce to proving that this group is zero when $\chi \not \equiv 1$. We use once again that $$ \cg_{\qp}/\hcal{I}_{\qp}^{\hrm{wild}} \simeq \left(\lim \limits_{\substack{\longleftarrow \\ r\geq 1, \, \hrm{norm \, map}}} \mathbb{F}_{p^r}^{\times} \right) \rtimes \varphi^{\widehat{\Z}}$$ where the limit is identified with $\hcal{I}_{\qp}/\hcal{I}_{\qp}^{\hrm{wild}}$.
			
			The inflation-restriction sequence for $\hcal{K}:=(\hcal{I}_{\qp}/\hcal{I}_{\qp}^{\hrm{wild}}) \rtimes \varphi^{\prod_{l\neq p} \mathbb{Z}_l}$ gives $$0 \rightarrow H^1(\varphi^{\zp}, \chi^{\hcal{K}}) \rightarrow H^1(\cg_{\qp}/\hcal{I}_{\qp}^{\hrm{wild}},\chi) \rightarrow H^1(\hcal{K},\chi_{|\hcal{K}})$$ As $\hcal{K}$ is pro-(prime to $p$), the last term vanishes. As $\zp$ is pro-$p$, we have $\chi_{|\varphi^{\zp}}\equiv 1$, our hypothesis $\chi \not \equiv 1$ implies $\chi_{|\hcal{K}} \not \equiv 1$. We deduce that $\chi^{\hcal{K}}=\{0\}$ and the first term also vanishes.
		\end{proof}
		
		\begin{coro}\label{no_nonsplit}
			Let $V$ be an irreducible object in $\hrm{Rep}_{\ov{\hbb{F}}_p} \,(\cg_{\qp}\times \cg_{\qp})$. Any $W\subset \hrm{Res}_{\cg_{\qp}}^{(\cg_{\qp}\times \cg_{\qp})} V$ which is an extension of two distincts characters is semisimple.
		\end{coro}
		\begin{proof}
			Such irreducible $V$ must be written $V_1 \boxtimes V_2$ where $V_i$ are irreducible objects in $\hrm{Rep}_{\ov{\hbb{F}}_p} \,\cg_{\qp}$. It is a general fact that any normal pro-$p$-subgroup of a profinite group acts trivially on any irreducible representation over a (inductive limit) of finite characteristic $p$ fields. Both $V_1$ and $V_2$ factorise by $\hcal{I}_{\qp}^{\hrm{wild}}$. Thus, Lemma \ref{triv_inertia_extension} applied to $\hrm{Res}_{\cg_{\qp}}^{(\cg_{\qp}\times \cg_{\qp})} V$ concludes.
		\end{proof}
		
		\vspace{0.5 cm}
		
		We are ready for our second result.

		\begin{prop}\label{failure_2}
			Let $\rho$ belong to $\hrm{Rep}_{\overline{\mathbb{F}}_p}\, \cg_{\qp}$, of length $3$ and Loewy length $3$, generic and good conjugate. Let $\Pi$ be weakly compatible with $\rho$.
			
			\begin{enumerate}
				\item We have $\hrm{dim} \mathbf{V}_{\Delta}(\Pi)\leq 5$.

				\item We have $\overline{L}^{\boxtimes}(\rho)^{\ord}=\mathbf{V}_{\Delta}(\Pi^{\ord}) \hookrightarrow \mathbf{V}_{\Delta}(\Pi)$.
				
				\item If $p\neq 2$, we have $ \hrm{dim} \mathbf{V}_{\Delta}(\Pi) \leq 4$, iff $\mathbf{V}_{\Delta}(\Pi) \rightarrow \mathbf{V}_{\Delta}(\Pi_{\ord})$ is nonzero.
			\end{enumerate}
			
			The first point implies that $\mathbf{V}_{\Delta,\xi}(\Pi)\not\simeq \overline{L}^{\otimes}(\rho)$.
	
		\end{prop}
		\begin{proof}
			We have an exact sequence $$0 \rightarrow \mathbf{V}_{\Delta}(\Pi^{\ord})\rightarrow \mathbf{V}_{\Delta}(\Pi) \rightarrow \mathbf{V}_{\Delta}(\Pi/ \Pi^{\ord}).$$ In particular, thanks to weak compatibility and Corollary \ref{inj_breuil_zab}, we have $\mathbf{V}_{\Delta,\xi}(\Pi^{\ord})\hookrightarrow \overline{L}^{\otimes}(\rho)$. Using the second bijection of Proposition \ref{bij_PS}, we get a labelled module diagram for $\mathbf{V}_{\Delta,\xi}(\Pi^{\ord})$ looking like \begin{center}
				\begin{tikzcd}
					& {\chi_1^2\chi_3} \\
					{\chi_1^2\chi_2} \\
					& {\chi_1\chi_2^2}
					\arrow["{>}"{marking, allow upside down}, no head, from=2-1, to=1-2]
					\arrow["{>}"{marking, allow upside down}, no head, from=2-1, to=3-2]
				\end{tikzcd}
			\end{center} In particular, genericity implies that it does not have $(\det\, \rho)$ as Jordan-Hölder component. Thanks to Corollary \ref{coro_subrep_restriction}, any nonzero subrepresentation of $\ov{L}^{\otimes}/\ov{L}^{\otimes}(\rho)^{\ord}$ has $\det\, \rho$ in its socle. We deduce that the image of $\mathbf{V}_{\Delta,\xi}(\Pi^{\ord})$ sits in $\ov{L}^{\otimes}(\rho)^{\ord}$. As they have same length, it is exactly $\ov{L}^{\otimes}(\rho)^{\ord}$.
			
			Consider $W:= \mathbf{V}_{\Delta,\xi}(\Pi)/\mathbf{V}_{\Delta,\xi}(\Pi^{\ord})$. Weak compatility and the above isomorphism show that it is a subrepresentation of $\overline{L}^{\otimes}(\rho) / \overline{L}^{\otimes}(\rho)^{\ord}$. We deduce that $\hrm{soc}_{\cg_{\qp}}(W)\rightarrow \hrm{soc}_{\cg_{\qp}}(\overline{L}^{\otimes}(\rho) / \overline{L}^{\otimes}(\rho)^{\ord})$ is an injection, hence $\hrm{soc}_{\cg_{\qp}}(W)$ is a direct sum of $(\det \rho)$. Left exactness of $\mathbf{V}_{\Delta,\xi}$ also implies that $\hrm{soc}_{\cg_{\qp}}(W)$ is a subrepresentation of $\mathbf{V}_{\Delta,\xi}(\Pi/\Pi^{\ord})$. This last representation sits in an exact sequence $$0\rightarrow \mathbf{V}_{\Delta,\xi}(\Pi_{\hrm{SS}}) \rightarrow \mathbf{V}_{\Delta,\xi}(\Pi/\Pi^{\ord}) \rightarrow \mathbf{V}_{\Delta,\xi}(\Pi_{\ord}).$$ Using the last part of Proposition \ref{bij_PS} to compute $\mathbf{V}_{\Delta,\xi}(\Pi_{\ord})$, $\det \rho$ is not one of its Jordan-Hölder factors, yielding that $\hrm{soc}_{\cg_{\qp}} W$ is a subrepresentation of $\mathbf{V}_{\Delta,\xi}(\Pi_{\hrm{SS}})$. As $\Pi_SS$ comes from an irreducible supersingular of $\mathbb{F}_q$ for $q$ big enough, the remark following Theorem \ref{irred_or_zero} says that $\mathbf{V}_{\Delta,\xi}(\Pi_{\hrm{SS}})$ is the restriction of an irreducible representation of $(\cg_{\qp}\times \cg_{\qp})$, hence $\hrm{soc}_{\cg_{\qp}} W$ is one dimensional by Lemma \ref{mult_one}. As $W$ is a subrepresentation of $\overline{L}^{\otimes}(\rho) / \overline{L}^{\otimes}(\rho)^{\ord}$, the description in Corollary \ref{coro_subrep_restriction} proves that it is at most $2$-dimensional which finishes to prove the first claim.
			
			It also proves that if $\dim W =2$, the representation $W$ is an extension of two distincts (thanks to genericity) characters. Using Corollary \ref{no_nonsplit}, it can't be a subrepresentation of $\mathbf{V}_{\Delta,\xi}(\Pi_{\hrm{SS}})$, which proves the third claim.
			
			Finally, use simultaneously the two first isomorphisms of Proposition \ref{bij_PS} to show that \linebreak$\mathbf{V}_{\Delta,\xi}(\Pi^{\ord})=\overline{L}^{\otimes}(\rho)^{\ord}$ implies $\mathbf{V}_{\Delta}(\Pi^{\ord})=\overline{L}^{\boxtimes}(\rho)^{\ord}$.
		\end{proof}
	
	\begin{rem}
		We used nothing more that weak compatibility. With (stronger) compatibility in the sense of \cite{breuil2021conjectures}, we would obtain directly that $\mathbf{V}_{\Delta,\xi}(\Pi_{\hrm{SS}})$ is $1$-dimensional. I don't see how it may improve the dimension bound.
	\end{rem}
	
	\begin{rem}
		Be careful that Lemma \ref{mult_one} breaks for $\# \Delta \geq 3$. In this case, one could strengthen the weak compatibility by replacing $\mathbf{V}_{\xi}(\Pi) \simeq \ov{L}^{\otimes}(\rho)$ by $\mathbf{V}_{\Delta}(\Pi) \subset \ov{L}^{\boxtimes}(\rho)$. This also give precise bounds of the dimension. For $\rho$ of dimension, length and Loewy length $4$, it gives $\dim \mathbf{V}_{\Delta}(\Pi)\leq 9$.
	\end{rem}

	
	
	
	
	\printbibliography

\end{document}